\documentclass[11pt]{article}

\usepackage[margin=1in]{geometry}
\usepackage{enumerate}
\usepackage{amsmath}
\usepackage{amssymb,latexsym}
\usepackage{amsthm}
\usepackage{color}
\usepackage{graphicx}
\usepackage{hyperref}
\usepackage{mathabx}

\date{}
\newcommand{\cA}{{\mathcal A}}

\newcommand{\cP}{{\mathcal P}}

\newcommand{\cD}{{\mathcal D}}
\newcommand{\cC}{{\mathcal C}}
\newcommand{\cL}{{\mathcal L}}

\newcommand{\F}{{\mathbb{F}}}
\newcommand{\cS}{{\cal S}}
\newcommand{\cT}{{\cal T}}

\newcommand{\PG}{\mathrm{PG}}
\newcommand{\AG}{\mathrm{AG}}

\newcommand{\Tr}{\mathrm{Tr}}

\newtheorem{theorem}{Theorem}[section]
\newtheorem{lemma}[theorem]{Lemma}
\newtheorem{result}[theorem]{Result}
\newtheorem{corollary}[theorem]{Corollary}
\newtheorem{definition}[theorem]{Definition}
\newtheorem{proposition}[theorem]{Proposition}
\newtheorem{example}[theorem]{Example}
\newtheorem{property}[theorem]{Property}
\newtheorem{notation}[theorem]{Notation}

\usepackage{fancyhdr}
\begin{document}

\title{Generalizing Korchm\'aros--Mazzocca arcs}
\author{Bence Csajb\'ok and Zsuzsa Weiner\thanks{
		The first author is supported by the J\'anos Bolyai Research Scholarship of the Hungarian Academy of Sciences and by OTKA Grant No. PD 132463. Both authors  acknowledge the support of OTKA Grant No. K 124950.}}

\maketitle

\begin{center}
\emph{We dedicate our work to the memory of our high school mathematics teacher, Dr.\ J\'anos Urb\'an to whom we are both very grateful.}
\end{center}

\begin{abstract}
In this paper, we generalize the so called Korchm\'aros--Mazzocca arcs, that is, point sets of size $q+t$ intersecting each line in $0, 2$ or $t$ points in a finite projective plane of order $q$. For $t\neq 2$, this means that each point of the point set is incident with exactly one line meeting the point set in $t$ points.

In $\mathrm{PG}(2,p^n)$, we change $2$ in the definition above to any integer $m$ and describe all examples when $m$ or $t$ is not divisible by $p$. We also study mod $p$ variants of these objects, give examples and under some conditions we prove the existence of a nucleus. 
\end{abstract}

\bigskip

{\it MSC2020 subject classification:} 51E20, 51E21

\section{Introduction}
A $(q+t)$-set $\cal K$ of type $(0,2,t)$ is a point set of size $q+t$ in a finite projective plane of order $q$ meeting each line in $0$, $2$ or in $t$ points. Note that if $t\neq 2$ then  this means that through each point of $\cal K$ there passes a unique line meeting $\cal K$ in $t$ points. 
For $t=1$ we get the ovals, for $t=2$ the hyperovals; thus this concept generalizes well-known objects of finite geometry. 
They were studied first by Korchm\'aros and Mazzocca in 1990, see \cite{KM}, that is why nowadays they are called KM-arcs. 
For $1<t<q$, they proved that KM-arcs exist only for $q$ even and $t \mid q$. KM-arcs have been studied mostly in Desarguesian planes, where G\'acs and Weiner proved that the $t$-secants of a KM-arc are concurrent \cite{GW}. For a different proof see \cite{BCs}. For various examples see \cite{Belga1, Belga2,GW, PVDesarg}.
Let $\Pi_q$ denote a (not necessarily Desarguesian) projective plane of order $q$. Examples of Vandendriessche \cite{PV} show that the $t$-secants of a KM-arc are not necessarily concurrent in $\Pi_q$.

In this paper, we generalize the concept of KM-arcs. We give examples and prove some characterization type results.

Throughout the paper, an $i$-secant will be a line intersecting our point set in $i$ points, the $1$-secants will be called tangents. An $i_p$-secant is a line intersecting our point set in $i \pmod p$ points. Sometimes we will need to distinguish between $i_p$-secants having $0$ points in common with our point set and $i_p$-secants intersecting our point set in at least a point. The second type of lines will be called proper $i_p$-secants. Many of our examples are related to subplanes of order $\sqrt q$ of a projective plane of order $q$; these are also called Baer subplanes. 

\medskip
\noindent
{\bf Definition \ref{defKMarcs}} \emph{A \emph{generalized KM-arc $\cS$ of type $(0,m,t)$} is a proper non-empty subset of points of size $q(m-1)+t$ in $\Pi_q$ meeting each line in $0$, $m$, or in $t$ points.}

\medskip

It is easy to see that when $t\neq m$, then each point of a generalized KM-arc $\cS$ of type $(0,m,t)$ in $\Pi_q$ is incident with exactly one $t$-secant and $q$ $m$-secants.

We also allow $m=t$, which gives the well-known maximal arcs. So in Desarguesian planes for $1<m=t<q$ they only exist for $q$ even (\cite{maxarc, maxarc2}). 

If $t=1$ (and $m\neq 1$) then generalized KM-arcs are called regular semiovals and G\'acs proved the following.

\begin{result}[\cite{gacssemi}]
\label{regsemi}
In $\PG(2,q)$, generalized KM-arcs of type $(0,m,1)$ (i.e. regular semiovals) are ovals $(m=2)$ and unitals $(m=\sqrt{q}+1)$. 
\end{result}

\medskip
\noindent
{\bf Definition \ref{defmodpKMarcs}} 
\emph{A \emph{mod $p$ generalized KM-arc $\cal S$ of type $(0, m, t)_p$} is a proper non-empty subset of points in $\Pi_q$, $q=p^n$, $p$ prime, such that each point $R \in \cal S$ is incident with a $t_{p}$-secant and the other $q$ lines through $R$ are $m_{p}$-secants, where $0\leq m,t \leq p-1$ are not necessarily distinct integers.} 

\medskip

The following theorems are the main results of our paper.

\medskip
\noindent
{\bf Theorem \ref{mainm0}}	\emph{Let $\cS$ be a mod p generalized KM-arc of type $(0, m, t)_p$ in $\PG(2,q)$, $q > 17$. Assume that $t\neq m$. If there are no $0$-secants of $\cS$ or $m=0$, then the $t_{p}$-secants of $\cS$ are concurrent.}
\medskip

\noindent
{\bf Theorem \ref{main}} \emph{For a generalized KM-arc $\cal S$ of type $(0,m,t)$ in $\PG(2,q)$, $q=p^n$, $p$ prime, either $m \equiv t \equiv 0 \pmod p$ or $\cal S$ is one of the following:
\begin{enumerate}[\rm(1)]
   \item a set of $t$ collinear points $(m=1)$,
   \item the union of $m$ lines incident with a point $P$, minus $P$ $(t=q)$,
   \item an oval $(t = 1,\, m=2)$,
   \item a maximal arc with at most one of its points removed $(t = m,\, t=m-1)$,
   \item a unital $(t = 1,\, m=\sqrt q+1)$.
 \end{enumerate}}
\medskip

The proofs rely on a stability result of Sz\H onyi and Weiner regarding $k$ mod $p$ multisets; and other polynomial techniques which ensure that in case of $t \not\equiv m \pmod p$ the $t_{p}$-secants meeting a fixed $m_{p}$-secant in $\cal S$ are concurrent, see Section \ref{renitentsection}. 
We also discuss connections with the Dirac--Motzkin conjecture regarding the number of lines meeting a point set of $\mathrm{PG}(2,\mathbb{R})$ in two points and a construction relying on sharply focused arcs of $\PG(2,q)$, see Section \ref{Diracsection}.

\medskip

Finally, we point out some relations with group divisible designs.
A $k$-GDD is a triple $({\cal V}, {\cal G}, {\cal B})$, where $\cal V$ is a set of points, $\cal G$ is a partition of $\cal V$ into parts (called groups), $|{\cal G}|>1$, and $\cal B$ is a family of $k$-subsets (called blocks) of $\cal V$ such that every pair of distinct elements of $\cal V$ occurs in exactly one block or in one group but not both. For more details and for the more general definition see \cite[Part IV]{Handbook}. 
If $t\neq m$, then the $t$-secants of a generalized KM-arc $\cS$ of type $(0,m,t)$ induce a partition on the points of $\cS$ and so it gives an $m$-GDD with the special property that each group in $\cal G$ has the same size $t$. Note that these GDDs are naturally embedded into a finite projective plane. Most probably the parameters of the GDDs coming from our examples on generalized KM-arcs are not new, but the explicit construction makes them interesting.

\section{Generalized KM-arcs}

\begin{definition}
\label{defKMarcs}
A \emph{generalized KM-arc $\cS$ of type $(0,m,t)$} is a proper non-empty subset of points of size $q(m-1)+t$ in $\Pi_q$ meeting each line in $0$, $m$, or in $t$ points.
\end{definition}

\begin{proposition}
 If $t \neq m$, then each point of a generalized KM-arc $\cS$ of type $(0,m,t)$ in $\Pi_q$ is incident with exactly one $t$-secant and $q$ $m$-secants.  \qed
\end{proposition}

In the introduction, we saw that ovals, maximal arcs and KM-arcs are generalized KM-arcs. Now let see some further examples, which we will refer to as \emph{trivial}:

\begin{example} \label{trivialwith0}
	Trivial examples for generalized KM-arcs of type $(0,m,t)$ admitting $0$-secants:
	\begin{enumerate}[\rm(1)]
		\item 	a set of $t\,(< q + 1)$ collinear points $(m=1)$,
		\item 	union of $m\,(< q+1)$ lines through a point $P$, minus $P$ $(t=q)$,
	    \item 	ovals $(t = 1,\, m=2)$,
		\item   a maximal arc with at most one of its points removed $(t = m,\, t=m-1)$. 
	\end{enumerate}
\end{example}

\begin{example} \label{trivialwithout0}
Trivial examples for generalized KM-arcs of type $(0,m, t)$ without $0$-secants:
	\begin{enumerate}[\rm(1)]
		\item 	a set of $q+1$ collinear points $(m=1)$,
		\item 	a unital $(t = 1,\, m=\sqrt q + 1)$,
		\item  complement of a Baer subplane
		$(t = q-\sqrt q,\, m=q)$,
		\item  complement of a point $(t = q,\, m=q+1)$.
	\end{enumerate}
\end{example} 

First we characterize  generalized KM-arcs without $0$-secants. 
Such sets intersect every line in $m$ or $t$ points; they are sets of type $(m,t)$.

A minimal $r$-fold blocking set $B$ is a point set intersecting every line in at least $r$ points such that each point of $B$ is incident with at least one $r$-secant of $B$. 

\begin{result}
[{\emph {\cite[Theorem 1.1]{minimultiple}}}]
\label{AnuragEtAll}
A minimal $t$-fold blocking set $B$ in a finite projective plane $\pi$ of order $n$ has size at most
\[\frac{1}{2}n\sqrt{4tn - (3t + 1)(t-1)} + \frac{1}{2}(t-1)n + t.\]
If $n$ is a prime power, then equality occurs exactly in the following cases:
\begin{enumerate}[\rm(1)]
    \item $t = n$ and $B$ is the plane $\pi$ with one point removed,
    \item $t = 1$, $n$ a square, and $B$ is a unital in $\pi$,
    \item $t = n- \sqrt n$, $n$ a square, and $B$ is the complement of a Baer subplane in $\pi$.
\end{enumerate}
\end{result}

A $1$-fold blocking set is also called a blocking set. The result above was already proved by Bruen and Thas (\cite{BT}) for blocking sets, showing that a minimal blocking set has size at most $n\sqrt n + 1$.

Clearly, if $t<m$ then generalized KM-arcs of type $(0,m,t)$ without $0$-secants are minimal $t$-fold blocking sets.

\begin{theorem}
\label{degencomb}
A generalized KM-arc $\cS$ of type $(0, m,t)$ without $0$-secants in $\Pi_q$, $q$ is a prime power, is always trivial, i.e. one of Example \ref{trivialwithout0}.
\end{theorem}
\begin{proof}
Note that $m\neq t$ since $\cS$ has to be a proper subset of $\Pi_q$. 
Let $k$ denote the size of any set of type $(m,t)$. Let $n_m$ denote the number of $m$-secants and $n_t$ denote the number of $t$-secants. Then 
\begin{equation}
n_m+n_t = q^2+q+1,
\end{equation}
\begin{equation}
mn_m+tn_t = (q+1)k,    
\end{equation}
\begin{equation}
m(m-1)n_m+t(t-1)n_t = k(k-1).
\end{equation}
From these equations one can easily deduce the following equations. For more details, see for example \cite{PentillaRoyle}. 

\begin{equation}
\label{Royle}
k^2-k(q(m+t-1)+m+t)+mt(q^2+q+1)=0.
\end{equation}
The number of $t$-secants incident with any point $Q\notin \cS$, using that $k=q(m-1)+t$, is
\begin{equation}
\label{Royle2}
    \frac{k-m(q+1)}{t-m}=1-\frac{q}{t-m}.
\end{equation}

This number must be a non-negative integer. Thus, if $t>m$, then $1-q/(t-m)=0$ and hence $t=q+1$ and $m=1$. This is only possible if $\cS$ is a line.

From now on we may assume $t<m$.
After substituting $k=t+q(m-1)$ in \eqref{Royle} and dividing by $q$, we obtain
\begin{equation}
\label{stand}
m^2-m t-m-q t+t^2=0.
\end{equation}
Then, since $t<m$,
\[m=\frac{1}{2} \left(\sqrt{4 q t-3 t^2+2 t+1}+t+1\right).\]
Then $\cS$ must be a minimal $t$-fold blocking set whose size $q(m-1)+t$ obtains the upper bound in Result \ref{AnuragEtAll} and hence the result follows.
\end{proof}

There are some more sophisticated examples, all of them with the property $m\equiv t \equiv 0 \pmod p$. 

\begin{example}[In terms of GDDs this was found by Wallis, see {\cite[Theorem 2.34]{Handbook}}. 
In $\PG(2,9)$ it is the same as {\cite[Example 4.4]{BBGSzW}} related to an extremal linear code.]
\label{Baer}
Let $\Pi_q$ be a projective plane of order $q$ and $\Pi_{\sqrt{q}}$ a Baer subplane of $\Pi_q$. 
Take any point $P$ of $\Pi_{\sqrt{q}}$ and denote by $\cL$ the union of the ${\sqrt{q}}+1$ lines of $\Pi_q$ which are incident with $P$ and meet $\Pi_{\sqrt{q}}$ in ${\sqrt{q}}+1$ points. Then the point set $\cL \setminus \Pi_{\sqrt{q}}$ is a generalized KM-arc of type $(0, \sqrt q, q-\sqrt q)$.	
\end{example}

Example \ref{Baer} exists in every finite projective plane admitting Baer subplanes. In Desarguesian planes, we can generalize this example. To see this we have to introduce some notation.
Let $f(x)$ be an $\F_q$-linear $\F_{q^n} \rightarrow \F_{q^n}$ function. The graph of $f$ is the affine point set 
\[U_f:=\{(x,f(x))\colon x\in \F_{q^n}\}\subseteq \AG(2,q^n).\]
The points of the line at infinity, $\ell_{\infty}$, are called directions. A direction $(d)$ is the common point of the lines with slope $d$. The set of directions determined by $f$ is:
\[D_f:=\left\{\left(\frac{f(x)-f(y)}{x-y}\right)\colon x,y\in \F_{q^n},\, x\neq y\right\}.\]

Since $f$ is $\F_q$-linear, for each direction $(d)$, there is a non-negative integer $e$, such that each line of $\PG(2,q^n)$ with slope $d$ meets $U_f$ in $q^e$ or $0$ points. The value $e$ will be called the \emph{exponent of $(d)$}.

\begin{example}
	\label{trace}

Put $f(x)=\Tr_{q^n/q}(x)=x+x^q+x^{q^2}+\ldots+x^{q^{n-1}}$.
Then $|D_f|=q^{n-1}+1$, the exponent of $(0)$ is $n-1$, the exponent of the points of $D_f \setminus \{(0)\}$ is $1$ and it is $0$ for the not determined directions. More precisely, $U_f\cup D_f$ is contained in 
\[\cL:=\ell_{\infty}\cup \bigcup_{y\in \F_q} \{(x,y) \colon x\in \F_{q^n}\},\]
which is the union of $q+1$ lines incident with $(0)$.

Then $\cL \setminus (D_f \cup U_f)$ is a generalized KM-arc of type $(0, q, q^n-q^{n-1})$ in $\PG(2, q^n)$. 
\end{example}

Note that when $n=2$, then Example \ref{trace} gives Example \ref{Baer} in Desarguesian planes.

The next example has only few $0$-secants, later it will turn out that in some sense this is an extreme example. 

\begin{result}[Mason {\cite[Theorem 2.5]{Mason}}]
\label{few0secants}
In $\PG(2,p^n)$, $p$ prime and $m < n$, there exist sets of type $(0,p^n-p^m,p^n-2p^m+1)$ and of size $(p^n-p^m)(p^n-1)$ with three $0$-secants.
\end{result}

\begin{example}
\label{genkmfew0}
When $p=3$ and $m=n-1$ then the point set of Result \ref{few0secants} is a generalized KM-arc of type $(0,2q/3,q/3)$ in $\PG(2,q)$, $q=p^n$, $p$ prime, with three $0$-secants and $2(q-1)$ $t$-secants. 
\end{example}

In the following extremal cases it is easy to characterize generalized KM-arcs. 

\begin{proposition}
\label{easy}
Let $\cS$ be a generalized KM-arc of type $(0,m, t)$  in $\Pi_q$. Then the following holds:
\begin{enumerate}[\rm(1)]
    \item if $t=q+1$, then $\cS$ is a line,
    \item if $t=q$, then $\cS$ is the union of $m$ concurrent lines, with their common point $P$ removed,
    \item if $m=q+1$, then $\cS$ is the complement of a point,
    \item if $m=q$ and $q$ is a prime power, then $\cS$ is the complement of a Baer subplane or $\cS$ is an affine plane of order $q$ with at most one of its points removed,
    \item if $m=1$, then $\cS$ is a subset of a line.
\end{enumerate}
\end{proposition}
\begin{proof}
We only prove $(4)$, the rest of them are straightforward (recall that by definition $\cS$ is a proper subset of $\Pi_q$).

If $\cS$ is a blocking set, then by Theorem \ref{degencomb} $\cS$ is the complement of a Baer subplane. Otherwise, denote by $\ell$ a $0$-secant of $\cS$ and suppose for the contrary that there exist two points $P, Q\notin \ell \cup \cS$. 
Since $|\cS|\geq q$, there is a point $R\in \cS \setminus PQ$. The lines $RP$ and $RQ$ are not $q$-secants of $\cS$ and hence both of them are $t$-secants incident with $R$, a contradiction. 
\end{proof}

Next we prove some combinatorial properties of a generalized KM-arcs.

\begin{lemma}
\label{combinatorial}
Let $\cS$ be a generalized KM-arc of type $(0,m, t)$  in $\Pi_q$. Then the following holds:
\begin{enumerate}[\rm(1)]
    \item $m\mid q(q-t)$,
    \item $\gcd(m,t) \mid q$,
    \item for any point $P\notin \cS$ if $t(P)$ denotes the number of $t$-secants incident with $P$ then $t(P)t \equiv t-q \pmod m$,
    \item $t \mid q(m-1)$,
    \item if $q(m-1)<(q+1-t)t$, then $m \mid q$.
    \item if $m,t\neq q$, $q=p^n$, $p$ prime, then the number of $0$-secants of $\cS$ is divisible by $p$,
    \item if $m \notdivides q-t$, then the $t$-secants of $\cS$ form a minimal blocking set of the dual plane.
\end{enumerate}
\end{lemma}

\begin{proof}
 Counting pairs $(P,\ell)$, $P\in \cS \cap \ell$ with $\ell$ an $m$-secant of $\cS$ gives 
 \[mN = q|\cS|=q^2m+qt-q^2,\] 
 where $N$ is the number of $m$-secants, and hence $(1)$ follows.
 
The lines incident with $P\notin \cS$ meet $\cS$ in a multiple of $\gcd(m,t)$ points and hence $\gcd(m,t)$ divides $|\cS|=qm+t-q$; proving $(2)$.

To prove $(3)$, note that the lines incident with $P\notin \cS$ meet $\cS$ in $0$, $t$, or in $m$ points. Let $m(P)$ denote the number of $m$-secants incident with $P$. Then $t(P)t + m(P) m = |\cS|=qm+t-q$ and hence $t(P)t \equiv t-q \pmod m$.

To see $(4)$, observe that the $t$-secants form a partition of the points in $\cS$ and hence $t\divides |\cS|$.

Consider a $t$-secant $\ell$ and suppose that each point of $\ell \setminus \cS$ is incident with a further $t$-secant. Then $q(m-1)=|\cS \setminus \ell|\geq (q+1-t)t$ since the $t$-secants of $\cS$ form a partition of $\cS$. If $q(m-1)<(q+1-t)t$ then it follows that there exists at least one point $P \notin S$ on each $t$-secant, such that the number of $t$-secants incident with $P$ is $1$. Then $(5)$ follows from $(3)$.

To prove $(6)$, note that the number of $0$-secants of $\cS$ is the total number of lines of $\Pi_q$ minus the number of $t$-secants, and the number of $m$-secants of $\cS$, that is,
\[q^2+q+1- \frac{q(m-1)+t}{t} -\frac{(q(m-1)+t)q}{m}.\]
If $m,t\neq q$ then this number is divisible by $p$.

When $(7)$ holds, then by $(2)$ $m\neq t$. 
Also, $m\notdivides q-t$ yields $m \notdivides |\cS|$ and hence points not in $\cS$ are incident with at least one $t$-secant. The minimality follows from the fact that points of $\cS$ are incident with a unique $t$-secant. 
\end{proof}

Let $\cS$ be a generalized KM-arc of type $(0,m,t)$ in $\Pi_q$, $q=p^n$, $p$ prime. When $\cS$ is not a blocking set and $m, t \neq q$, then by Lemma \ref{combinatorial} $(6)$ the number of $0$-secants of $\cS$ is at least $p$ and hence Example \ref{genkmfew0} is extremal in this sense. Also, if the $t$-secants of $\cS$ do not form a blocking set of the dual plane, then $m \divides q-t$. 
Example \ref{genkmfew0} is extremal also in this sense, since there $m=q-t$. We are grateful to Tam\'as H\'eger for finding Example \ref{genkmfew0} in $\PG(2,9)$ which led us to find the paper of Mason.

\begin{theorem}
For a generalized KM-arc $\cS$ of type $(0,m,t)$  in $\Pi_q$, if $m \notdivides q-t$, then $\cS$ is either a maximal arc with one point removed or there are more than $q+1$ $t$-secants and hence they cannot be concurrent. 
\end{theorem}

\begin{proof}
By Lemma \ref{combinatorial} the $t$-secants of $\cS$ form a minimal blocking set and hence their number is at least $q+1$ with equality if and only if they are concurrent. In this case $|\cS|=(q+1)t=t+q(m-1)$, thus $m-1=t$ and hence by adding the common point of $t$-secants to $\cS$ we obtain a maximal arc.
\end{proof}

\section{\texorpdfstring{Mod $p$ generalized KM-arcs of type $(0, m, t)_p$}{Mod p generalized KM-arcs of type (0,m,t)p}}

In this section we generalize further the concept of KM-arcs.
\begin{notation}
Recall that a line is a \emph{$t_{p}$-secant} if it meets $\cal S$ in $t \pmod p$ points. Recall also that a $t_{p}$-secant is \emph{proper} if it meets $\cal S$ in at least $1$ point. We defined \emph{$m_{p}$-secants} and \emph{proper} \emph{$m_{p}$-secants} similarly.
\end{notation}

\begin{definition}
\label{defmodpKMarcs}
A \emph{mod $p$ generalized KM-arc $\cal S$ of type $(0, m, t)_p$} is a proper non-empty subset of points in $\Pi_q$, $q=p^n$, $p$ prime, such that each point $R \in \cal S$ is incident with a $t_{p}$-secant and the other $q$ lines through $R$ are $m_{p}$-secants, where 
the integers $m$ and $t$ are not necessarily distinct and 
$0\leq m,t \leq p-1$.
\end{definition}

Generalized KM-arcs of type $(0, m, t)$ are of course mod $p$ generalized KM-arcs of type $(0, m', t')_p$ as well, where $m'$ and $t'$ are integers satisfying 
$m \equiv m' \pmod p$, $t \equiv t' \pmod p$ and $0\leq m',t' \leq p-1$. Now let us see some further examples.

\begin{definition}
\label{holvagy}
For $0 \leq c \leq p-1$, a $c$ mod $p$ intersecting  point set/multiset is a point set/multiset with the property that each line which intersects it in at least $1$ point, intersects it in $c$ mod $p$ points. (Intersection number calculated with multiplicity.) 
Note that $c$ mod $p$ intersecting point sets and mod $p$ generalized KM-arcs of type $(0,c,c)_p$ are the same objects.
\end{definition}

One can easily construct $c$ mod $p$ intersecting  point sets (or multisets).
Linear sets are $1$ mod $p$ intersecting  point sets (see \cite{OlgaSurvey}), the union of $c'$ linear sets is a $c$ mod $p$ intersecting point set or multiset where 
$c \equiv c' \pmod p$ with $0\leq c \leq p-1$. 

Let $L_1$ and $L_2$ be $0$ mod $p$ intersecting  point sets. If $L_2\subseteq L_1$, then $L_1\setminus L_2$ is also a $0$ mod $p$ intersecting  point set. 
Similarly, we get $c$ mod $p$ intersecting point sets with $c \equiv c_1-c_2 \pmod p$, $0\leq c \leq p-1$, when $L_1$ is $c_1$,  $L_2$ is $c_2$ mod $p$ intersecting point set and lines meeting $L_1$ meet $L_2$ as well.

Here are some examples for mod $p$ generalized KM-arcs of type $(0, m, t)_p$ with $t \neq m$.

\begin{example}
\label{modpkmarc}
 A $c$ mod $p$ intersecting point set with one of its points removed is a mod $p$ generalized KM-arc of type $(0,c,d)_p$ with $d\equiv c-1 \pmod p$. 
 Note that the proper $d_p$-secants of this point set are concurrent.
\end{example}

Let ${\cal C}_1$ be a $c_1$ mod $p$ intersecting point set and ${\cal C}_2$ be a $c_2$ mod $p$ intersecting point set with exactly one common point. 
Assume that every line meets either both or none of the sets $\cC_1$ and $\cC_2$. Then the sum of ${\cal C}_1$ and  ${\cal C}_2$ is a $c$ mod $p$ intersecting multiset with $c \equiv c_1 + c_2 \pmod p$ and with  exactly one point with multiplicity different from $1$.

\begin{example}
\label{modpkmarc2}
 Let $\cal C$ be a $c$ mod $p$ intersecting multiset, such that only one point $Q \in {\cal C}$ has multiplicity $r$ and the rest of the points in $\cal C$ have multiplicity $1$, $p>r>0$. Then by deleting $Q$,
 we get a mod $p$ generalized KM-arc of type $(0, c, d)_p$ with $d\equiv c-r \pmod p$. Note that the proper $d_p$-secants of this point set are concurrent.
\end{example}

The sum of a unital or a Baer subplane (or even any small minimal blocking set) and one of its tangents are examples for point sets $\cC$ in Example \ref{modpkmarc2}. There exist more sophisticated examples as well, in \cite{AK} the authors construct a multiset meeting each line in $\sqrt q-1$ or $2\sqrt q -1$ points in $\PG(2,q)$, $q$ square. This multiset has a unique point with multiplicity greater than $1$, its multiplicity is $q-1$. By removing this point we obtain a mod $p$ generalized KM-arc of type $(0,p-1,0)_p$. Note that the proper $0_p$-secants of this point set are concurrent.

\begin{lemma}
\label{1modpthroughP}
    Let $\cal S$ be a mod $p$ generalized KM-arc of type $(0, m, t)_p$ where $t \neq m$.  Take $Q \notin \cS$. If there is no $0$-secant incident with $Q$ or $m=0$, then the number of $t_{p}$-secants incident with $Q$ is $1$ mod $p$. 
\end{lemma}
\begin{proof}
    The conditions imply that $t_{p}$-secants incident with $Q$ are proper. 
    If $t_p(Q)$ denotes the number of $t_{p}$-secants incident with $Q$, then we get
    \[t_p(Q) t + (q+1-t_p(Q))m\equiv t \pmod p,\]
    \[(t_p(Q)-1)(t-m)\equiv 0 \pmod p,\]
    and hence $t_p(Q) \equiv 1 \pmod p$.
\end{proof}

\begin{proposition}
\label{numtp}
    Let $\cal S$ be a mod $p$ generalized KM-arc of type $(0, m, t)_p$ where $t \neq m$. Then the number of proper $t_{p}$-secants is at most $q\sqrt{q}+1$. 
\end{proposition}
\begin{proof}
By Lemma  \ref{1modpthroughP}, the $0$-secants and the $t_{p}$-secants form a blocking set on the dual plane. The proper $t_{p}$-secants in this blocking set are essential and hence their number is at most $q\sqrt{q}+1$ (see \cite{BT}).
\end{proof}

\subsection{\texorpdfstring{The $c$ mod $p$ intersecting case}{The c mod p intersecting case}}
\label{cmodpsection}

\begin{proposition}[{\cite[Lemma 3]{semi} for $c=1$ and \cite[Exercise 13.4]{Sziklaipoly} for $c$ in general}]
\label{exercise}
A $c$ mod $p$ intersecting point set $\cS$ either meets every line in $c$ mod $p$ points or $c=1$ and $|\cS|\leq q-p+1$.
\end{proposition}
\begin{proof}
If $\cS$ does not have $0$-secants, or if $c=0$, then 
$\cS$ meets each line in $c$ mod $p$ points; hence the result follows. So we may assume that $\cS$ is an affine point set and $1\leq c \leq p-1$. Identify $\AG(2,q)$ with $\F_{q^2}$. Note that three points are collinear if and only if for the corresponding elements $a,b,c$, we have $(a-b)^{q-1}=(a-c)^{q-1}$ (see for example \cite{Sziklaipoly}).
Define
\[f(X):=\sum_{s\in \cS}(X-s)^{q-1}.\]

Counting points of $\cS$ on lines incident with a point of $\cS$ gives $|\cS|\equiv c \pmod p$ and hence the degree of $f$ is $q-1$.
For $s\in \cS$ we have $f(s)=(c-1)\sum_{e^{q+1}=1}e=0$, thus $|\cS|\leq q-1$ and hence $|\cS|\leq q-p+c$ since this is the largest integer smaller than $q-1$ and congruent to $c$ mod $p$. Point sets of size less than $q+2$ have tangents, thus it follows that $c=1$.
\end{proof}

For mod $p$ generalized KM-arcs this gives the following result.

\begin{proposition}
\label{prop0b}
	If for a mod $p$ generalized KM-arc $\cS$ of type $(0, m, t)_p$, $t=m$ holds, then $t=m \in \{0,1\}$ or $\cS$ cannot have $0$-secants. 
\end{proposition}

\begin{proposition}
\label{prop0a}
If for a mod $p$ generalized KM-arc $\cS$ of type $(0, m, t)_p$, $t=m$ holds, then $t = m = 0$, or $\cS$ is a set of $t$ collinear points, or $\cS$ is a unital.
\end{proposition} 
\begin{proof}
If $\cS$ has no $0$-secants then the result follows from Theorem \ref{degencomb}.

If $\cS$ has $0$-secants, then by Proposition \ref{prop0b}, we may assume $t = m = 1$. By Proposition \ref{exercise}, $|\cS|\leq q-1$ and hence each point of $\cS$ is incident with at least $3$ tangents. It follows that $m=1$ and hence $\cS$ is a set of $t$ collinear points.
\end{proof}

\section{Further generalization}

In this section, we generalize further the concept of KM-arcs. 

\medskip

Throughout this section, $\cal A$ will be a proper subset of $\Pi_q$, $q=p^n$, with the following property.
For each point $R\in \cal A$,
there exist integers $0 \leq m_R,\, t_R \leq p-1$ such that there is {\it at most} one line which is incident with $R$ and meets $\cal A$ in $t_R$ mod $p$ points and the other lines incident with $R$ meet $\cal A$ in $m_R$ mod $p$ points. Points of $\cal A$ incident with exactly one $t_R$ mod $p$ secant and with $q$ $m_R$ mod $p$ secants (and hence $t_R \neq m_R$) will be called \emph{regular}, the other points of $\cal A$ will be called \emph{irregular}. If $R$ is regular, then the unique line incident with $R$ and meeting $\cal A$ in $t_R$ mod $p$ points will be called \emph{renitent}. 

\medskip

Note that we get back the definition of a mod $p$ generalized KM-arc if $m_R$ and $t_R$ do not depend on the choice of the point $R\in \cal A$.
However, it will turn out that for regular points these values do not depend on the choice the point. 

\begin{proposition}
	\label{propa}
	If $Q$ is regular then $t_Q \equiv |{\cal A}| \pmod p$. If $Q$ is irregular then $m_Q \equiv |{\cal A}| \pmod p$. 
\end{proposition}
\begin{proof}
    It follows by counting the points of ${\cal A}$ on the lines incident with $Q$.
\end{proof}

\begin{theorem}
\label{verygeneral}
	For the point set $\cal A$, one of the following holds:
	\begin{enumerate}[\rm(1)]
		\item Each point of $\cal A$ is regular. Then for any two points $P,R \in \cal A$ it holds that $t_P=t_R$ and $m_P=m_R$, i.e. $\cal A$ is a mod $p$ generalized KM-arc of type $(0, m, t)_p$ with $m\neq t$.
		\item  Each point of $\cal A$ is irregular and hence $\cal A$ is a $c$ mod $p$ intersecting point set, cf. Definition \ref{holvagy} and Section \ref{cmodpsection}.
		\item There is a unique irregular point $Q$ and the renitent lines are incident with this point. In this case $\cA \setminus \{Q\}$ is as in  (1) or (2) and in the former case the proper $t_{p}$-secants are concurrent.
	\end{enumerate}
\end{theorem}

\begin{proof}
Let $a$ be an integer so that  $0\leq a \leq p-1$ and $|\cA|\equiv a \pmod p$.
If $\cA$ is a subset of a line, then $\cA$ is as in Case (1) (if $a\neq 1$) or as in Case (2) (if $a= 1$); thus from now on we may assume that $\cA$ contains three points in general position.

If each point is regular then by Proposition \ref{propa}, there exists $t$ such that renitent lines at the points of $\cal A$ are incident with $t$ mod $p$ points of $\cal A$. For $P,R\in \cal A$ either $|PR \cap {\cal A}| \not\equiv t \pmod p$ and hence $m_P=m_R$, or $PR$ is the unique renitent line incident with $P$ and with $R$. 
Take a point $Q\in {\cal A} \setminus PR$. The number of points of $\cal A$ in $QP$ and in $QR$ is not congruent to $t$ mod $p$, thus they are both congruent to $m_Q$ mod $p$, thus $m_P=m_R$.

Suppose that the points $Q_1$ and $Q_2$ are irregular.  
Then $m_{Q_1}=m_{Q_2}=t_{Q_1}=t_{Q_2}=a$. 
By the first paragraph, we may assume that there exists $P \in \cA \setminus Q_1Q_2$. We show that $P$ must be irregular. Since $|PQ_1 \cap \cA|  \equiv |PQ_2 \cap \cA| \pmod p$, it follows that $m_P=a$ as one of $PQ_1$ or $PQ_2$ is not renitent at $P$. Also $t_P = a$ by Proposition \ref{propa}. 
Starting from the two irregular points $P$ and $Q_1$ the same argument shows that also the points of $\cA \cap Q_1Q_2$ are irregular.
Thus all points are irregular and 
hence $\cA$ is a $|\cA|$ mod $p$ intersecting point set.

On the other hand if there is a unique irregular point $Q$,
then each line incident with this point is an $a$ mod $p$
secant. Also, by Proposition \ref{propa}, for any other (regular) point $P$, $t_P=a$. Hence all renitent lines pass through $Q$.
Finally, we prove $m_{P_1}=m_{P_2}$ for any two regular points. 
If $Q\notin P_1P_2$ then it is straightforward. If $Q\in P_1P_2$ then take a regular point $P_3\notin P_1P_2$. Then $Q\notin P_1P_3 \cup P_2P_3$ and hence $m_{P_1}=m_{P_3}$ and $m_{P_2}=m_{P_3}$.
After removing $Q$, either all regular points turn to be irregular, or all of them remain regular in this new point set.
\end{proof}

\section{Renitent lines are concurrent}
\label{renitentsection}

In this section, our aim is to prove that the $t_{p}$-secants of a mod $p$ generalized KM-arc $\cS$ of type $(0, m, t)_p$ meeting a fixed $m_{p}$-secant in $\cS$ are concurrent, when $t\neq m$.

Now we again define renitent lines in a very similar context.

\begin{definition}
Let $\cT$ be a point set of $\AG(2,q)$, $q=p^n$, $p$ prime. The line $\ell$ with slope $d$ is said to be renitent w.r.t. $\cT$ if there exists an integer $\mu$ such that $|\ell \cap \cT|\not\equiv \mu \pmod p$ and $|r \cap \cT| \equiv \mu \pmod p$ for each line $r\neq \ell$ with slope $d$.
\end{definition}

The next result can be viewed a generalization of \cite[Theorem 5]{semi}, see also \cite[Proposition 2]{seminuclear} and \cite[Remark 7]{dirszonyi}.

\begin{lemma}[Lemma of renitent lines]
	\label{renitent}
Let $\cT$ be a point set of $\AG(2,q)$, $2<q=p^n$, $p$ prime, such that $|\cT| \not\equiv 0 \pmod p$.
Then the renitent lines w.r.t. $\cT$ are concurrent.
\end{lemma}
\begin{proof}
For each $0\leq \mu \leq p-1$ we define the subset of directions $\cD_{\mu} \subseteq \ell_{\infty}$ in the following way: a direction $(d)$ is in $\cD_{\mu}$ if and only if there are exactly $q-1$ affine lines with direction $(d)$ such that each of them meets $\cT$ in $\mu$ mod $p$ points.
First we show that the renitent lines with slope in $\cD_{\mu}$ are concurrent. It will turn out that their point of concurrency depends only on $\cT$ and not on $\mu$. Thus each of the renitent lines will be incident with this point. For the sake of simplicity we will say `renitent line', instead of `renitent line with slope in $\cD_{\mu}$'.

Suppose $\cD_{\mu} \neq \emptyset$ and put $s=|\cT|$, then $s \equiv (q-1)\mu+\tau\equiv \tau-\mu \pmod p$, where each renitent line meets $\cT$ in $\tau\equiv s+\mu$ points modulo $p$ for some $0\leq \tau \leq p-1$. Note that $\tau\neq \mu$. If $|\cD_{\mu}|<q+1$, then we can always assume $(\infty) \notin \cD_{\mu}$.
If $|\cD_{\mu}|=q+1$, then it is enough to prove that renitent lines with slope in $\cD_{\mu} \setminus (\infty)$ are concurrent. Indeed, if we prove this, then after a suitable affinity we get that any $q$ of the $q+1$ renitent lines are concurrent. Since $q>2$, the result then follows for all renitent lines.

Let $\cT=\{(a_i,b_i)\}_{i=1}^{s}$ and
\[H(U,V):=\prod_{i=1}^{s}(U+a_iV-b_i)=\sum_{j=0}^{s}h_j(V)U^{s-j},\]
that is, the R\'edei polynomial of $\cT$. Here $h_j(V)$ is a polynomial of degree at most $j$.
Note that $h_0(V)=1$ and $h_1(V)=AV-B$, where $A=\sum_{i=1}^{s}a_i$ and $B=\sum_{i=1}^{s}b_i$.
For each $d\in \F_q$, $U=k$ is a root of $H(U,d)$ with multiplicity $r$ if and only if the line with equation
$Y=dX+k$ meets $U$ in exactly $r$ points.
Let $(0,a(d))$ be the intersection of the line $X=0$ and the unique renitent line through $(d)\in \cD_{\mu}$.
Then the lines incident with $(d)$ yield
\[H(U,d)=(U-a(d))^{\alpha_d p + \tau}\prod_{w\in \F_q\setminus \{a(d)\}}(U-w)^{\beta_{w,d}\, p + \mu},\]
with $\alpha_d p+\tau+(q-1)\mu+\sum_{w\in \F_q \setminus \{a(d)\}} \beta_{w,d}\, p=s$, for some $\alpha_d, \beta_{w,d}\in \F_q$. Multiplying both sides by $(U-a(d))^{p+\mu-\tau}$ yields
\[H(U,d)(U-a(d))^{p+\mu-\tau}=(U-a(d))^{(\alpha_d+1) p + \mu}\prod_{w\in \F_q\setminus \{a(d)\}}(U-w)^{\beta_{w,d}\, p + \mu}.\]
Here the right-hand side can be written as
\[(U^q-U)^{\mu} f(U^p),\]
for some polynomial $f$. The degrees at both sides are $s+p+\mu-\tau$.
The second greatest degree on the right-hand side is at most $s+\mu-\tau$. 
Hence the coefficient of $U^{s+p+\mu-\tau-1}$ is zero on the left-hand side, i.e. \[h_1(d)-(p+\mu-\tau)a(d)=0.\]
Since $\tau \neq \mu$, it follows that $a(d)=h_1(d)/(\mu-\tau)=-h_1(d)/s=(B-Ad)/s$. 
Note that $a(d)$ does not depend on the choice of $\mu$.
It follows that $Y=dX+(B-Ad)/s$ is the equation of the renitent line through $(d)$. 
For $d\in \F_q$, these lines are concurrent, their common point is $(A/s,B/s)$.
\end{proof}

\subsection{Easy consequences of the Lemma of Renitent lines}
\begin{proposition}
	\label{sok}
	If $t \neq m$ holds for a mod $p$ generalized KM-arc $\cS$ of type $(0, m, t)_p$ in $\PG(2,q)$, then for any $m_{p}$-secant $\ell$ the $t_{p}$-secants incident with the points of $\ell \cap \cS$ are concurrent.
\end{proposition}
\begin{proof}
We may consider $\ell$ as the line at infinity and so $\cT:=\cS \setminus \ell$ is an affine point set in the affine plane $\PG(2,q) \setminus \ell$.
Since $|\cT| \equiv t-1+(q-1)(m-1)\equiv t-m \not\equiv 0 \pmod p$, we can apply Lemma \ref{renitent}.
\end{proof}

The next propositions are easy corollaries of the proposition above.

\begin{proposition}
	\label{prop2} For a generalized KM-arc $\cS$ of type $(0, m, t)$ in $\PG(2,q)$,
	if $1<t<q$ and $t \not\equiv m \pmod p$, then $m \mid q$.
\end{proposition}
\begin{proof}
	It follows from Proposition \ref{sok} that for each $P\notin \cS$, if $P$ is incident with more than one $t$-secant then it is incident with at least $m$ $t$-secants. 
	Consider a $t$-secant $\ell$. If there is a point of $\ell \setminus \cS$ incident with a unique $t$-secant ($\ell$), then by part (3) of Lemma \ref{combinatorial}
    $m \mid q$. If there is no such point, then each $P\in \ell \setminus \cS$ is incident with at least $m-1$ $t$-secants other than $\ell$.
    Then the number of $t$-secants other than $\ell$ is at least $(q+1-t)(m-1)$. On the other hand the number of $t$-secants different from $\ell$ is $|\cS|/t - 1 = q(m-1)/t$.
    It follows that
	\[(q+1-t)(m-1)t\leq q(m-1),\]
	a contradiction, when $m > 1$.
\end{proof}

\begin{lemma}
\label{q-1}
If $t\neq m$ holds for a mod $p$ generalized KM-arc $\cS$ of type $(0, m, t)_p$ in $\PG(2,q)$, then either the proper $t_{p}$-secants pass through a common point or 
for each $P\notin \cS$ it holds that $|\{Q \colon QP \mbox{ is a $t_{p}$-secant }\}\cap \cS|\leq q-1$.
\end{lemma}

\begin{proof}
Assume that the proper $t_{p}$-secants do not pass through a common point.
Let $P$ be a point not in $\cS$ and let $l_1, l_2, \dots, l_k$ denote the proper $t_{p}$-secants through $P$.
The proper $t_{p}$-secants are not concurrent, which yields that there is a point, say $R$, which is in $\cS$ but not on the lines $l_i$.
Hence the line $PR$ must be an $m_{p}$-secant. So the points of $\cS$ on the lines $l_i$ must lie on the $q-1$ lines $r_1, r_2, \ldots, r_{q-1}$ through $R$, which are different from $PR$ and from the unique $t_{p}$-secant through $R$. The line $PR$ is an $m_{p}$-secant and so by Proposition \ref{sok}, on each of the lines $r_1,r_2,\ldots, r_{q-1}$, we may see at most one point of $\cS \cap \{ l_1 \cup l_2 \ldots \cup l_k \}$ and hence the proposition follows.
\end{proof}

Then the next theorem follows immediately.

\begin{theorem}
For a mod $p$ generalized KM-arc $\cS$ of type $(0, m, t)_p$ in $\PG(2,q)$ 
assume $t \neq m$ and assume also that the proper $t_{p}$-secants are not concurrent.
Let $t'$ and $m'$ be the least number of  $\cS$ points on a proper $t_{p}$-secant and on a proper $m_{p}$-secant, respectively.
Then the number of proper $t_{p}$-secants through a point $ P \not \in \cS$ is at most $(q-1) / t'$. Hence the number of points on an $m_{p}$-secant is also at most  $(q-1) / t'$.
 \qed
\end{theorem}

\section{Characterization type results}

In this section, we will prove some characterization results on mod $p$ generalized KM-arcs of type $(0, m, t)_p$. In the special case of generalized KM-arcs, our result will be stronger. First recall some earlier stability results on $k$ mod $p$ sets. 

\begin{property}[{\cite[Property 3.5]{SzW}}]
\label{most}
Let $\cal M$ be a multiset in ${\rm PG}(2,q)$, $q=p^n$, where $p$ is prime. Assume that there are $\delta$ lines that intersect $\cal M$ in not $k$ mod $p$ points. 
We say that Property \ref{most} holds if for every point $Q$ incident with more than $q/2$ lines meeting $\cal M$ in not $k$ mod $p$ points, there exists a value $r\not\equiv k \pmod p$ such that more than $2\frac{\delta}{q+1}+5$ of the lines through $Q$ meet $\cal M$ in $r$ mod $p$ points. 
 \end{property}
 
\begin{result}[{\cite[Theorem 3.6]{SzW}}]
\label{general}
Let $\cal M$ be a multiset in ${\rm
PG}(2,q)$, $17< q$, $q=p^n$, where $p$ is prime.
Assume that the number of lines intersecting $\cal M$ in not
$k$ mod $p$ points is $\delta$, where $\delta <
(\lfloor \sqrt q \rfloor +1)(q+1-\lfloor \sqrt q \rfloor )$. 
Assume furthermore, that Property \ref{most} holds.
Then there
exists a multiset $\cal{M}'$ with the property that it intersects every
line in $k$ mod $p$ points and 
the number of points whose modulo $p$ multiplicity is different in
${\cal M}$ than in ${\cal M'}$ is exactly $\left\lceil \frac{\delta}{q+1}\right\rceil$.
\end{result}

\begin{corollary}
\label{spectrum}
Let $\cal M$ be a multiset in $\PG(2,q)$, $17< q$, $q=p^n$, where $p$ is prime. 
Assume that the number of lines intersecting $\cal M$ in not
$k$ mod $p$ points is $\delta < 4q-8$ and that Property \ref{most} holds.
Then Result \ref{general} can be applied and it yields
\[\delta\in \{0\} \cup \{q+1\}\cup \{2q,2q+1\} \cup \{3q-3,\ldots,3q+1\}.\]
\end{corollary}

\begin{result}[{\cite[Result 2.1,\, Remark 2.4,\, Lemma 2.5 (1)]{SzW}}]
\label{smallbigindex}
Let ${\cal M}$ be a multiset in $\PG(2,q)$, $17 < q$, so that the number of lines intersecting it in not $k$ mod $p$ points is 
$\delta$. Then
the number $s$ of not $k$ mod $p$ secants through any point of ${\cal M}$ satisfies $qs - s(s - 1) \leq \delta$.
\end{result}

\subsection{\texorpdfstring{When most of the lines are $m_{p}$-secants}{When most of the lines are mp-secants}}

In this section, we will consider mod p generalized KM-arcs of type $(0, m, t)_p$ in PG$(2,q)$. We will be able to characterize such an arc, when most of the lines intersect it in $m \pmod p$ points. 

\medskip
\noindent 
From now on, let $\cS$ be a mod p generalized KM-arc of type $(0, m, t)_p$ in $\PG(2,q)$ and assume that  $m \neq t$ and $\cS$ has no $0$-secants or $m=0$. So all $t_{p}$-secants are proper $t_{p}$-secants. Assume also that $q > 17$.

 Note that in this case, the lines that intersect $\cS$ in not $m$ mod $p$ points are exactly the $t_{p}$-secants; hence Property \ref{most} holds. The next lemma is an easy consequence of Proposition \ref{numtp} and Result \ref{smallbigindex}.

\begin{lemma}
\label{tpsecants}
The number of $t_{p}$-secants through a point is either at most $\lfloor \sqrt{q} \rfloor + 2$ or at least $q-\lfloor \sqrt{q} \rfloor -1$. \qed
\end{lemma}

\begin{lemma}
\label{largeindexpointexists}
There is always at least one point (not in $\cS$), through which there pass at least $q-\lfloor \sqrt{q} \rfloor -1$ $t_{p}$-secants.
\end{lemma}

\begin{proof}
First suppose that the number of $t_{p}$-secants, $\delta$, is less than $(\lfloor \sqrt q \rfloor +1)(q+1-\lfloor \sqrt q \rfloor )$. Then by Result \ref{general}, there is a 
point set $\cP$ of size $\left\lceil \frac{\delta}{q+1}\right\rceil < \sqrt q+1$ such that adding the points of $\cP$ with the right non zero modulo $p$ multiplicities we obtain a multiset $\cS'$ meeting every line in $m$ mod $p$ points.
This means that through a point $P \in {\cal P}$ there pass at
most $|{\cal P}|-1$ $m_{p}$-secants and hence at 
least $q+1 - (|{\cal P}|-1)$ $t_{p}$-secants. 
Since $|{\cal P}| < \sqrt q +1$,  $P$ is a point incident with lots of $t_{p}$-secants. Hence the points of $\cal P$ are not in $\cS$.

Next assume that the number of $t_{p}$-secants is at least $(\lfloor \sqrt q \rfloor +1)(q+1-\lfloor \sqrt q \rfloor )$. 
The $t_{p}$-secants partition the points of $\cS$ and each of them contains at least one point of $\cS$, thus
\[|S| \geq (\lfloor \sqrt q \rfloor +1)(q+1-\lfloor \sqrt q \rfloor ).\]

On the contrary, assume that there is no point with at least  $q-\lfloor \sqrt{q} \rfloor -1$ $t_{p}$-secants on it.
It follows from Proposition \ref{sok} and Lemma \ref{tpsecants}, that each $m_{p}$-secant contains at most $\lfloor \sqrt{q} \rfloor + 2$ points from $\cS$.
So $|\cS| \leq q(\lfloor \sqrt{q} \rfloor + 1 ) + t_{min}$, where $t_{min}$ is the least number of points from $\cS$ on a $t_{p}$-secant. If $t_{min} > 1$, then the number of $t_{p}$-secants is at most $q( \lfloor \sqrt{q} \rfloor + 1 ) / 2 + 1$ and we have a contradiction. So $t_{min} = 1$ and 
\[t=1.\]
If the $m_{p}$-secants contain at most $\lfloor \sqrt{q} \rfloor$ points, then $|\cS| \leq q\lfloor \sqrt{q} \rfloor - q +1$ and again we have a contradiction. If there is an $m_{p}$-secant  $e$ with $\lfloor \sqrt{q} \rfloor + 2$ points, then by Proposition \ref{sok}, there is a point $N$ incident with at least $\lfloor \sqrt{q} \rfloor + 2$ $t_{p}$-secants. By Lemma \ref{tpsecants} and by the assumption that there is no point with at least  $q-\lfloor \sqrt{q} \rfloor -1$ $t_{p}$-secants on it, the number of $t_{p}$-secants through $N$ must be exactly $\lfloor \sqrt{q} \rfloor + 2$. By Lemma \ref{1modpthroughP}, $\lfloor \sqrt{q} \rfloor + 2 \equiv 1 \pmod p$ and so $m=1$. This contradicts the assumption that $m \neq t$, since now 
$t = 1$ too.

Hence all $m_{p}$-secants contain at most $\lfloor \sqrt{q} \rfloor + 1$ points from $\cS$ and there exists a line $\ell$ with exactly $\lfloor \sqrt{q} \rfloor + 1$ points from $\cS$. Let $M$ be the point through which the $t_{p}$-secants of $\ell$ pass. The number of $t_{p}$-secants through a point is congruent to $1=t \neq m$ mod $p$, hence through $M$ there pass exactly $\lfloor \sqrt{q} \rfloor + 2$ $t_{p}$-secants. On the rest of the $q-1 - \lfloor \sqrt{q} \rfloor$ not $t_{p}$-secants through $M$, we see at most $(q-1 - \lfloor \sqrt{q} \rfloor)(\lfloor \sqrt{q} \rfloor + 1)$ points of $\cS$, so there are at most this many $t_{p}$-secants not incident with $M$. Hence the total number of $t_{p}$-secants is at most $\lfloor \sqrt{q} \rfloor + 2 + (q-1 - \lfloor \sqrt{q} \rfloor)(\lfloor \sqrt{q} \rfloor + 1)$, which is again a contradiction.
\end{proof}

\begin{lemma}
\label{numofnonmp}
The number of $t_{p}$-secants is at most $2q + 1 + (\lfloor \sqrt{q} \rfloor + 2)^2$.
\end{lemma}

\begin{proof}
By Lemma \ref{largeindexpointexists}, there exists a 
point $M$ with at least $q-\lfloor \sqrt{q} \rfloor -1$ $t_{p}$-secants through it. 

First suppose that there are no more points incident with at least $q-\lfloor \sqrt{q} \rfloor -1$ $t_{p}$-secants. Let us count the number of points of $\cal S$ on the lines through $M$. On each of the $m_{p}$-secants through $M$, we see at most $\lfloor \sqrt{q} \rfloor + 2$ points by Proposition \ref{sok} and Lemma \ref{tpsecants}. And so by Lemma \ref{q-1}, in total $\cS$ has at most $(q-1) + (\lfloor \sqrt{q} \rfloor + 2)^2$ points. This is also an upper bound on the number of $t_{p}$-secants of $\cS$; hence we are done.

Now assume that there is another point, say $N$, with at least $q-\lfloor \sqrt{q} \rfloor -1$ $t_{p}$-secants through it. For the points in $\cS$, the unique $t_{p}$-secant through them pass either through $M$ or $N$ or it is skew to these two points. There are at most $(\lfloor \sqrt{q} \rfloor + 2)^2$ points $P$, so that neither $PM$ nor $PN$ is a $t_{p}$-secant. So the number of $t_{p}$-secants not through $M$ or $N$ is also at most this many. Hence the total number of $t_{p}$-secants is at most 
$2q + 1 + (\lfloor \sqrt{q} \rfloor + 2)^2 $. 
\end{proof}

The next proposition follows from Result \ref{general}, from Corollary \ref{spectrum} and from Lemma \ref{numofnonmp}.

\begin{proposition}
\label{extendable}
There exists a point set $\cal N$ of size at most $3$, so that if we add the points from $\cal N$ with multiplicity $m-t$ to $\cS$, we obtain a multiset intersecting each line in $m$ mod $p$ points. Consequently, the following properties hold for $\cal N$.
\begin{enumerate}[\rm(1)]
    \item a line contains $1$ mod $p$ point from $\cal N$ if and only if it is a $t_{p}$-secant,
    \item through a point in $\cal N$ there pass at least $q - 1$ $t_{p}$-secants of $\cS$,
    \item through a point not in $\cal N$ there pass at most $3$ $t_{p}$-secants.
\end{enumerate}
\qed
\end{proposition}

\begin{theorem}
\label{mainm0}	Let $\cS$ be a mod p generalized KM-arc of type $(0, m, t)_p$ in $\PG(2,q)$, $q > 17$. Assume that $t\neq m$. If there are no $0$-secants of $\cS$ or $m=0$, then the $t_{p}$-secants are concurrent.
\end{theorem}

\begin{proof}
Consider the point set $\cal N$ from Proposition \ref{extendable}. 

If $|{\cal N}| = 1$, then Proposition \ref{extendable} $(1)$ finishes the proof.

Assume that the points of $\cal N$ lie on a line $\ell$
and $|{\cal N}| > 1$. If there was a point of $\cS$ outside $\ell$, then by Proposition \ref{extendable} $(1)$ through this point there would pass at least two $t_{p}$-secants; a contradiction. Hence $\cS \subset \ell$, $m=1$ and $\ell$ is the only $t_{p}$-secant; again we are done. 

So we may assume that ${\cal N} = \{ N_1, N_2, N_3 \}$. 
From above, the points of ${\cal N}$ form a triangle. Let $P$ be a point in $\cS$ and not on the lines $N_iN_j$. Then by Proposition \ref{extendable}, $PN_1$, $PN_2$ and $PN_3$
are $t_{p}$-secants, so there are at least three $t_{p}$-secants through $P$; a contradiction. Hence the points of $\cS$ lie on the lines $N_1N_2$, $N_2N_3$ and $N_1N_3$. Each of the $t_{p}$-secants contains exactly 1 point from $\cS$, so $t \equiv 1 \pmod p$. Also, again by Proposition \ref{extendable} and by the current setting the number of $t_{p}$-secants through $N_1$ is $|\cS \cap N_2N_3|$.
$N_2N_3$ must be an $m_{p}$-secant (again by Proposition \ref{extendable} $(1)$), so by Lemma \ref{1modpthroughP}, $m$ is also 1 mod $p$; which contradicts our assumption.
\end{proof}

The theorem above yields a stronger characterization result on  generalized KM-arcs  of type $(0,m,t)$.

\begin{theorem}
	\label{main}
A generalized KM-arc $\cS$ of type $(0,m,t)$ in $\PG(2,q)$, $q=p^n$, $p$ prime, is either trivial, i.e. it is as in Examples \ref{trivialwith0} and \ref{trivialwithout0}, or $m \equiv t \equiv 0 \pmod p$.  
\end{theorem}

\begin{proof}
Assume $p \notdivides m$ or $p\notdivides t$. Then by Proposition \ref{prop0a}, we may assume that $t\not\equiv m \pmod p$ and by Proposition \ref{prop2} $t=1$ or $t\geq q$, or $m\mid q$.  
In the first case, as we mentioned before, G\'acs proved that the only examples are the ovals and unitals, cf. Result \ref{regsemi}. If $t=q$, then take a $t$-secant $\ell$ of $\cS$ and let $P$ be the unique point of $\ell \setminus \cS$. Since each point of $\cS$ is incident with a unique $t$-secant, all $t$-secants pass through $P$. 
If $t=q+1$ then there is a unique $t$-secant and hence $\cS$ is a line. 
If $m=1$, then $\cS$ is a $t$-subset of a line.

If $m>1$ and $m\mid q$, then from Theorem \ref{mainm0} either $p\mid t$ or the $t_{p}$-secants are concurrent. By Lemma \ref{1modpthroughP} the $t_{p}$-secants form a dual blocking set and so when $p\notdivides t$, there are exactly $q+1$ of them. In this latter case, $|\cS|=(q+1)t=q(m-1)+t$. So $m=t+1$, hence by adding the common point of the $t$-secants to $\cS$ we obtain a maximal arc.
\end{proof}

\section{More examples}
\subsection{Cone construction}
The construction method described is 
\cite{GW} can be used to construct mod $p$ generalized KM-arcs in $\PG(2,q^h)$ from mod $p$ generalized KM-arcs in $\PG(2,q)$. Start from a generalized KM-arc of type $(0, m, t)$ in $\PG(2,q)$, which admits the property that the $t$-secants go through the point $N$, or start from a maximal arc and a point $N$ not in the arc. 
In both cases if $N$ plays the role of $Q$ in \cite[Construction 3.3]{GW} then we get a generalized KM-arc of type $(0, m, tq^{h-1})$ in $\PG(2,q^h)$.
(For more details see \cite[Construction 3.3]{GW} and the proceeding paragraph.)
Similarly, starting from a mod $p$ generalized KM-arc of type $(0, m, t)_p$ in $\PG(2,q)$, which admits the property that the proper $t_{p}$-secants are concurrent, 
or start from a $m$ mod intersecting point set we may obtain a mod $p$ generalized KM-arc of type $(0, m, 0)_p$ in $\PG(2,q^h)$.

In both cases, when $t\neq m$, the construction yields examples with concurrent $t$-secants (in case of generalized KM-arcs) and concurrent proper $t_{p}$-secants (in case of mod $p$ generalized KM-arcs).

\subsection{Examples from the real projective plane}
\label{Diracsection}

In this section we consider generalized and mod $p$ generalized KM-arcs of $\PG(2,\mathbb{R})$ defined analogously as in finite projective planes. 
It is easy to see that any finite subset of a line is a generalized KM-arc. We will need the following two results.

\begin{result}[Sylvester--Gallai theorem]
Given a finite number of points in the Euclidean plane, either all the points lie on a single line, or there is at least one line which contains exactly two of the points.
\end{result}

\begin{result}[Melchior's inequality \cite{Melchior}]
Denote by $\tau_k$ the number of $k$-secants of a given point set $\cP$ of size at least $3$ in the Euclidean plane. 
If the points of $\cP$ are not collinear then $\tau_2 \geq 3+\sum_{k\geq 4}(k-3)\tau_k$.
\end{result}

\begin{proposition}
	\label{embedd}
	Let $\cP$ be a finite mod $p$ generalized KM-arc of type $(0,m,t)_p$ in $\PG(2,\mathbb{R})$ not contained in a line. Then $p=2$, $t=0$ and $m=1$.
\end{proposition}
\begin{proof}
	Denote by $\tau_k$ the number of $k$-secants of $\cP$ and put $n=|\cP|$. 
	Clearly, each point of $\cP$ is incident with more than one tangent and hence $m=1$. By the Sylvester--Gallai theorem $\cP$ will have $2$-secants, and hence $t=2$. Thus the number of proper $t_{p}$-secants of $\cP$ is at most $n/2$ and this yields also $\tau_2 \leq n/2$. 
	
	Next we show $p=2$ (and hence $t=0$). Again from the Sylvester-Gallai theorem, it can be easily shown by induction that $n\geq 3$ points of $\PG(2,\mathbb{R})$, not all of them collinear, span at least $n$ lines, i.e. $\sum_{k \geq 2} \tau_k \geq n$. 
	If $p > 2$, then $\cP$ cannot have $3$-secants, thus by 
	Melchior's inequality
	\[\tau_2 \geq 3+\sum_{k \geq 4}\tau_k=3+\sum_{k \geq 2}\tau_k - \tau_2\geq 3+n-\tau_2\]
	and hence $\tau_2 \geq n/2+3/2$, a contradiction.
\end{proof}

The following corollary can be deduced easily from above.
\begin{corollary}
The finite generalized KM-arcs of $\PG(2,\mathbb{R})$ are the finite subsets of  lines.
\end{corollary}

Suppose that there exists an injective map $\varphi$ from the points of a mod $2$ generalized KM-arc $\cP$ of type $(0,1,0)_2$ in $\PG(2,\mathbb{R})$ to $\PG(2,q)$, $q$ even, such that any triplet of points $Q,R,S \in \cP$ is collinear if and only if $\varphi(Q), \varphi(R), \varphi(S)$ are collinear.
The $2$-secants of a real point set $\cP$ are usually called ordinary lines. The Dirac-Motzkin conjecture, proved by Green and Tao \cite{GT}, is the following: If $n$ is large enough, then any $n$-set of $\PG(2,\mathbb{R})$, not all of them collinear, spans at least $n/2$ ordinary lines. On the other hand, if the embedded point set $\varphi(\cP)$ is a mod $2$ generalized KM-arc, then the number of even secants of $\cP$ is at most $n/2$. Hence it is exactly $n/2$ and thus $n$ is even. Up to projectivities, there is a unique known example, due to B\"or\"oczky, of $n$-sets determining exactly $n/2$ ordinary lines: a regular $m$-gon in $\AG(2,\mathbb{R})$ together with the $m$ directions determined by them, where $m=n/2$. For embeddings of regular $m$-gons, preserving parallelism of its secants, see the survey \cite{affreg} on affinely regular $m$-gons. Note that these objects all give rise to sharply focused arcs defined below.

\begin{definition}
	A $k$-arc of $\AG(2,q)$ is called sharply focused if it determines $k$ directions and it is called hyperfocused if it determines $k-1$ directions. 
\end{definition}

\begin{example}
\label{sarc}
In $\AG(2,q)$, $q$ even, consider a sharply focused arc $\cS$ of size $k$, $k$ odd. If $\cD$ denotes the set of $k$ directions determined by $\cS$, then $\cS \cup \cD$ is a mod $2$ generalized KM-arc of type $(0,1,0)_2$.
\end{example}

In Example \ref{sarc} the number of tangents to $\cS$ meeting $\cD$ is $k$.
Also, since $k$ is odd, each point of $\cD$ is incident with a unique tangent to $\cS$. Then Lemma \ref{renitent} applied to the affine point set $\cS$ gives that these $k$ tangents are concurrent, they meet in a point $R\notin \cS \cup \cD$. 
Note that $\cS \cup \{R\}$ is a hyperfocused arc determining the same set of directions as $\cS$. For $q$ even (and $k$ even or odd) the extendability of a sharply focused $k$-arc to a hyperfocused $(k+1)$-arc was proved by Wettl \cite{Wettl}.

\section*{Acknowledgement}
The authors are grateful to the anonymous referees for their valuable comments and suggestions which have certainly improved the quality of
the manuscript.

\begin{flushleft}
Bence Csajb\'ok \\
MTA--ELTE Geometric and Algebraic Combinatorics Research Group\\
ELTE E\"otv\"os Lor\'and University, Budapest, Hungary\\
Department of Geometry\\
1117 Budapest, P\'azm\'any P.\ stny.\ 1/C, Hungary\\
{{\em csajbokb@cs.elte.hu}}

\medskip

Zsuzsa Weiner\\
MTA--ELTE Geometric and Algebraic Combinatorics Research Group,\\
1117 Budapest, P\'azm\'any P.\ stny.\ 1/C, Hungary\\

{{\em zsuzsa.weiner@gmail.com}}\\

and\\
 
Prezi.com\\
 H-1065 Budapest, Nagymez\H o utca 54-56, Hungary\\
\end{flushleft}

\end{document}